\documentclass[11pt]{article}
\usepackage{amsthm, amsmath, amssymb, amsfonts, url, booktabs, tikz, setspace, fancyhdr, enumerate}
\usepackage[margin = 1in]{geometry}

\usepackage{hyperref}

\definecolor{refkey}{gray}{.75}
\definecolor{labelkey}{gray}{.2}

\usepackage[textsize=scriptsize,colorinlistoftodos]{todonotes}

\newtheorem{theorem}{Theorem}[section]
\newtheorem{proposition}[theorem]{Proposition}
\newtheorem{lemma}[theorem]{Lemma}
\newtheorem{corollary}[theorem]{Corollary}
\newtheorem{conjecture}[theorem]{Conjecture}

\theoremstyle{definition}

\newtheorem{example}[theorem]{Example}

\theoremstyle{remark}

\newcommand{\overbar}[1]{\mkern 1.5mu\overline{\mkern-1.5mu#1\mkern-1.5mu}\mkern 1.5mu}

\newcommand{\x}{\times}

\newcommand{\Hom}{\mathrm{Hom}}
\newcommand{\EE}{\mathbb{E}}

\newcommand{\FF}{\mathcal{F}}

\renewcommand{\S}{\mathcal{S}}

\newcommand{\W}{\mathcal{W}}
\newcommand{\TT}{\mathcal{T}}

\newcommand{\mockalph}[1]{}

\tikzstyle{p}+=[fill=black, circle, minimum width = 1pt, inner sep =
1pt]
\tikzstyle{w}+=[fill=white, draw, circle, minimum width = 1pt, inner sep =
1.5pt]

\title{\vspace{-0.7cm}Domination inequalities and dominating graphs}
\author{David Conlon\thanks{Department of Mathematics, California Institute of Technology, Pasadena, CA 91125, USA. Email: {\tt dconlon@caltech.edu}.
Research supported by NSF Award DMS-2054452.}
\and
Joonkyung Lee\thanks{
		Department of Mathematics, Yonsei University, Seoul and Extremal Combinatorics and Probability Group,
Institute for Basic Science (IBS), Daejeon, South Korea.
		Email: \texttt{joonkyunglee}@\texttt{yonsei.ac.kr}.
		Research supported by the National Research Foundation of Korea (NRF) Grant 2022R1C1C1010300, Samsung STF Grant SSTF-BA2201-02 and IBS-R029-C4.}
}
\date{}

\begin{document}
\maketitle

\begin{abstract}
We say that a graph $H$ dominates another graph $H'$ if the number of homomorphisms from $H'$ to any graph $G$ is dominated, in an appropriate sense, by the number of homomorphisms from $H$ to $G$. We study the family of dominating graphs, those graphs with the property that they dominate all of their subgraphs. It has long been known that even-length paths are dominating in this sense and a result of Hatami implies that all weakly norming graphs are dominating. In a previous paper, we showed that every finite reflection group gives rise to a family of weakly norming, and hence dominating, graphs. Here we revisit this connection to show that there is a much broader class of dominating graphs.
\end{abstract}

\section{Introduction}

A \emph{graphon} is a symmetric measurable function $W$ from $[0,1]^2$ to $[0,1]$, where symmetric here means that $W(x,y) = W(y,x)$ for all $(x,y) \in [0,1]^2$. Very roughly, this may be seen as a continuous analogue of 
a graph. The \emph{homomorphism density} $t_H(W)$ of a graph $H$ in a graphon $W$ is then given by 
\[t_H(W) = \EE \left[\prod_{ij \in E(H)} W(x_i, x_j)\right] = \int_{[0,1]^{v(H)}} \prod_{ij \in E(H)} W(x_i, x_j) \ d\mu^{v(H)},\]
where $\mu$ is the Lebesgue measure on $[0,1]$. Our concern in this paper will be with the following question: for which ordered pairs of graphs $(H, H')$ is it the case that $H$ \emph{dominates} $H'$, in the sense that
\[t_H(W)^{1/e(H)} \geq t_{H'}(W)^{1/e(H')}\]
for all graphons $W$? 

This question has been studied both implicitly and explicitly for decades. For instance, a celebrated conjecture of Sidorenko~\cite{S93, S932} says that $H$ dominates $K_2$ if and only if $H$ is bipartite. The necessity of the bipartiteness condition is simple to verify, but its sufficiency is a difficult problem which has attracted a great deal of attention in recent years~\cite{CFS10, CKLL, CL19, C22, CR21, KLL16, LiSz, Sz14}.

One of the oldest results in this direction is a result of Godsil that first appeared in a paper of Erd\H{o}s and Simonovits~\cite{ES82}. If we write $P_n$ for the path with $n$ vertices, his result says that $P_n$ dominates $P_m$ whenever $2 \leq m \leq n$ and $n$ is odd. That is, even-length paths dominate all of their connected subgraphs and this easily extends to all subgraphs. In what follows, we will refer to a graph $H$ with this property, that $H$ dominates $H'$ for all subgraphs $H'$ of $H$ with at least one edge, as a \emph{dominating graph}.

The only other graphs known to have this property are the weakly norming graphs whose study was initiated by Lov\'asz~\cite{L08} and Hatami~\cite{H10}. 
The fact that all weakly norming graphs are dominating was proved by Hatami~\cite{H10}. 
Until quite recently, only a handful of weakly norming graphs were known~\cite{H10, L12}, but a much broader collection of examples was found by the authors~\cite{CL17}, who showed how to associate a family of weakly norming graphs to any finite reflection group. The purpose of this paper is to show that the family of dominating graphs is broader still. In particular, we will show that both previous families of examples, weakly norming graphs and even-length paths, can be placed within a common framework.

One of the main results in~\cite{CL17} says that if there is a sequence of what we might call reflections, each of which takes a subgraph $J$ of a graph $H$ and produces another subgraph $J'$, 
that starts with a single edge and ends with $H$, then $H$ is weakly norming. The corresponding result here allows a second operation, which we call relocation, and says that if, through a sequence of reflections and relocations, one can start with a single edge and get to a graph $H$, then $H$ is dominating. 
In~\cite{CL17}, our results can be neatly packaged in terms of what we call reflection graphs. Unfortunately, there seems to be no such clean statement here.  
Instead, we will apply our reflection/relocation mechanism to produce 
a variety of examples,
each derived in some way from a reflection group. For now, we illustrate our results with just one family of examples, which, by a recent result of Sidorenko~\cite{Sunpub}, are known to not be weakly norming. Recall that the {\it $1$-subdivision} of a graph $H$ is the graph obtained by replacing each edge of $H$ by a path with two edges.

\begin{theorem} \label{thm:examp}
The $1$-subdivision of the complete bipartite graph $K_{t,t}$ is dominating.
\end{theorem}

The rest of the paper is organised as follows. In the next section, we prove some necessary conditions for a graph $H$ to dominate another graph $H'$, thereby also giving some conditions that must be satisfied by any dominating graph. In Section~\ref{sec:perc}, we describe our reflection/relocation mechanism and show that it produces dominating graphs. We then use this mechanism in Section~\ref{sec:examples} to find new examples of dominating graphs, including those mentioned in Theorem~\ref{thm:examp}. We talk about applications of the domination property in Section~\ref{sec:apps}, before concluding with some further remarks and open problems.

\section{Necessary conditions for domination} \label{sec:nec}

As we remarked in the introduction, any dominating graph must be bipartite. Here we establish some more necessary conditions.
First, we show that any connected dominating graph must be $1$-balanced, in the sense that $\frac{e(H)}{v(H) - 1} \geq \frac{e(H')}{v(H') - 1}$ for every subgraph $H'$ of $H$. The proof is essentially due to Hatami~\cite{H10}, who proved an analogous result for weakly norming graphs.

\begin{proposition}\label{prop:1-balanced}
Let $H$ and $H'$ be connected graphs.
If $H$ dominates $H'$, then $\frac{e(H)}{v(H) - 1} \geq \frac{e(H')}{v(H') - 1}$.
\end{proposition}

\begin{proof}
For $1 \leq k \leq n$, let $I_{n,k}$ be the square-shaped subset $\{(x,y):(k-1)/n<x,y<k/n\}$ of the unit square and 
let $W$ be the block graphon given by $W =\sum_{k=1}^n \mathbf{1}_{I_{n,k}}$. 
Then, since both $H$ and $H'$ are connected, 
\begin{align*}
    t_H(W) = \sum_{k=1}^n n^{-v(H)} = n^{-(v(H)-1)}~\text{ and }~t_{H'}(W) = \sum_{k=1}^n n^{-v(H')} = n^{-(v(H')-1)}.
\end{align*}
Thus, the domination inequality $t_H(W)^{1/e(H)}\geq t_{H'}(W)^{1/e(H')}$ can be rewritten as
\[n^{-(v(H)-1)/e(H)} \geq n^{-(v(H')-1)/e(H')},\]
which yields the desired inequality $\frac{e(H)}{v(H) - 1} \geq \frac{e(H')}{v(H') - 1}$.
\end{proof}

\begin{corollary} \label{cor:bal}
If $H$ is a connected dominating graph, then $\frac{e(H)}{v(H) - 1} \geq \frac{e(H')}{v(H') - 1}$ for every subgraph $H'$ of $H$.
\end{corollary}

\begin{proof}
By Proposition~\ref{prop:1-balanced}, we may assume that $H'$ is not connected.
 Let $F_1,F_2,\dots,F_t$ be the connected components of $H'$ and let $d_i:=\frac{e(F_i)}{v(F_i)-1}$.
Then
\begin{align*}
    \frac{e(H')}{v(H')-1} = \frac{\sum_i d_i(v(F_i)-1)}{\sum_i (v(F_i) -1/t)} 
    < \frac{\sum_i d_i(v(F_i)-1/t)}{\sum_i (v(F_i) -1/t)}, 
\end{align*}
and, hence, $\frac{e(H')}{v(H')-1}$ is smaller than a convex combination of the $d_i$. Therefore, there must be a component $F_j$ with $d_j>\frac{e(H')}{v(H')-1}$. Hence, since, by Proposition~\ref{prop:1-balanced}, $\frac{e(H)}{v(H) - 1} \geq \frac{e(F)}{v(F) - 1}$ for every connected $F$, we have $\frac{e(H)}{v(H) - 1}\geq d_j$, which implies that $\frac{e(H)}{v(H) - 1} \geq \frac{e(H')}{v(H') - 1}$ for every subgraph $H'$. 
\end{proof}

As pointed out in~\cite{GHL22}, this result is not true if $H$ is not connected, as may be seen by considering the weakly norming, and hence dominating, graph consisting of two disjoint copies of $K_{1,2}$ and the subgraph consisting of a single copy of $K_{1,2}$.

Our second condition states that every connected dominating graph has the property that every vertex on the smaller side of the bipartition has the same degree. The proof is again similar to a result of Hatami~\cite{H10}, who showed that every connected weakly norming graph is bi-regular, i.e., all vertices on the same side of the bipartition have the same degree. This stronger property does not necessarily hold for dominating graphs, since, as we shall see, there are many examples, including even-length paths and the example $C_6^+$ in Figure~\ref{fig:base}, which are not bi-regular.

\newdimen\R
\R=2cm

\begin{figure}
\centering
\begin{tikzpicture}
        \draw (0:\R) \foreach \x in {60,120,...,359} {
                -- (\x:\R)
            }-- cycle (90:\R) node[above] {} ;
        \draw (0:0)--(0:\R);
        \draw (0:0)--(120:\R);     
        \draw (0:0)--(240:\R);          
\end{tikzpicture}
\caption{The graph $C_6^+$, which is dominating but not weakly norming.} \label{fig:base}
\end{figure}
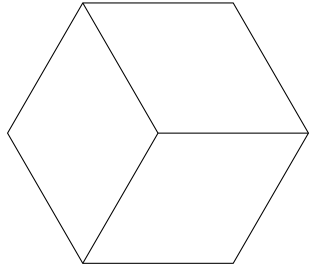

\begin{proposition}\label{prop:regular}
Let $H$ be a connected dominating graph with bipartition $A\cup B$. If $|A|\leq|B|$, then every vertex $v\in A$ must have the same degree. Moreover, the maximum degree $\Delta$ of $H$ is attained by the vertices in $A$.
\end{proposition}

\begin{proof}
Let $\varepsilon<1/2$ be positive and let $I$ and $J$ be the two disjoint intervals $[0,\varepsilon)$ and
$[\varepsilon,1]$, respectively. Denote by $W$ the graphon $\mathbf{1}_{I\times J}+\mathbf{1}_{J\times I}$. Denote by $S$ the star with $\Delta$ leaves, which is clearly a subgraph of $H$. Then
\begin{align*}
    t_H(W) = \varepsilon^{|A|}(1-\varepsilon)^{|B|}+\varepsilon^{|B|}(1-\varepsilon)^{|A|}~~\text{ and }~~t_S(W) = \varepsilon(1-\varepsilon)^{\Delta}+\varepsilon^{\Delta}(1-\varepsilon).
\end{align*}
Since $\varepsilon<1/2$, $t_S(W)\geq \varepsilon/2^{\Delta}$, so the inequality $t_H(W)\geq t_S(W)^{e(H)/\Delta}$, which holds since $H$ is dominating, gives
\begin{align*}
    2^{-e(H)}\varepsilon^{e(H)/\Delta}\leq t_S(W)^{e(H)/\Delta}\leq t_H(W)\leq 2\varepsilon^{|A|}.
\end{align*}
But this gives a contradiction for $\varepsilon$ sufficiently small 
unless $|A|\leq e(H)/\Delta$. On the other hand, $e(H)\leq \Delta |A|$. Therefore, $e(H)=\Delta|A|$, which means that every vertex in $A$ has degree $\Delta$.
\end{proof}

As a corollary, we note that if the two sides of the bipartition of a connected dominating graph $H$ have the same size, i.e., $|A|=|B|$, then the graph must be regular.

The reason we have largely focused on connected dominating graphs is because of the following result, which says that the components of any dominating graph must all be the same. This generalises, and gives a simpler proof of, a similar result for weakly norming graphs proved by Garbe, Hladk\'y and Lee~\cite{GHL22}.

\begin{lemma}
A graph without isolated vertices is dominating if and only if each of its components is the same dominating graph.
\end{lemma}

\begin{proof}
It easy to check that the disjoint union of multiple copies of the same dominating graph is dominating, so we will focus on the opposite direction. To this end, 
let $H$ be a dominating graph with no isolated vertices and at least two components and write $H = H_1 \cup H_2$, where $H_1$ is a single component and $H_2$ is the union of the remaining components. Then $t_H(W)^{1/e(H)} \geq t_{H_1}(W)^{1/e(H_1)}$ is equivalent to $t_{H_2}(W)^{1/e(H_2)} \geq t_{H_1}(W)^{1/e(H_1)}$, where we used that $t_H(W)= t_{H_1}(W) t_{H_2}(W)$. Similarly, $t_H(W)^{1/e(H)} \geq t_{H_2}(W)^{1/e(H_2)}$ is equivalent to $t_{H_1}(W)^{1/e(H_1)} \geq t_{H_2}(W)^{1/e(H_2)}$. Hence, if $H$ is dominating, we must have that $t_{H_1}(W)^{1/e(H_1)} = t_{H_2}(W)^{1/e(H_2)}$ for all $W$. But this is the same as saying that $t_{H_1}(W)^{e(H_2)} = t_{H_2}(W)^{e(H_1)}$ for all $W$, which in turn implies that $t_{e(H_2) H_1}(W) = t_{e(H_1) H_2}(W)$ for all $W$. But, by~\cite[Theorem 5.29]{L12}, this implies that $e(H_2) H_1 = e(H_1) H_2$, so $H_2$ must be a union of copies of $H_1$. To verify that $H_1$ is itself dominating, suppose that $H = rH_1$ and consider the subgraph $H' = rH'_1$, where $H'_1$ is a subgraph of $H_1$. Then, since $H$ is dominating,
\[t_{H_1}(W)^{1/e(H_1)} = t_{rH_1}(W)^{1/re(H_1)} = t_H(W)^{1/e(H)} \geq t_{H'}(W)^{1/e(H')} = t_{rH'_1}(W)^{1/re(H'_1)} = t_{H'_1}(W)^{1/e(H'_1)}\]
for all $W$, implying that $H_1$ is indeed dominating.
\end{proof}

\section{Layered percolation} \label{sec:perc}

Given a connected graph $H$, we say that an automorphism $\phi$ of $H$ is a \emph{cut involution} if the following conditions hold:

\begin{enumerate}

\item[(i)] $\phi$ is an involution, that is, $\phi = \phi^{-1}$;

\item[(ii)] the fixed point set $F_\phi = \{v \in V(H) : \phi(v) = v\}$ is a vertex cut of $H$;

\item[(iii)] $V(H) \setminus F_\phi$ is the disjoint union of sets $L_\phi$ and $R_\phi$, where there are no edges between $L_\phi$ and $R_\phi$ and they are mapped to one another under $\phi$.\footnote{This last condition was not stated explicitly in our earlier work~\cite{CL17}, though it was used implicitly throughout.}
\end{enumerate}
The \emph{cut involution group} $W_H$ is then the group of automorphisms of $H$ generated by its cut involutions.

Suppose now that a particular choice for $L_\phi$ and $R_\phi$ has been made for each cut involution $\phi$ of $H$. We define the \emph{left-folding map} $\phi^+: V(H) \rightarrow V(H)$ of a cut involution $\phi$ of $H$ by 
\begin{align*}
	\phi^+(v)=
	\begin{cases}
		\phi(v) ~&\text{ if }v\in R_\phi\\
		v &\text{ if }v\in L_\phi\cup F_\phi
	\end{cases}
\end{align*}
and, similarly, define the \emph{right-folding map} $\phi^-$ by swapping the roles of $L_\phi$ and $R_\phi$. For brevity, we refer to both the left- and right-folding maps as \emph{half-folding maps} of $\phi$. The \emph{conjugate} $\overbar{\psi}$ of a half-folding map $\psi$ is then the other half-folding map defined by the same cut involution.

If $J$ is an edge subset of $H$ and $\phi$ is a cut involution of $H$, we define two edge sets $J^+(\phi)$ and $J^-(\phi)$ by
\begin{align*}
	J^+(\phi)=\{e\in E(H):\phi^+(e)\in J\}\text{ and }
	J^-(\phi)=\{e\in E(H):\phi^-(e)\in J\}.
\end{align*}
That is, $J^+(\phi)$ is the graph formed by copying the edges of $J$ from the left half onto the right half, while $J^-(\phi)$ copies the edges from the right half onto the left half. The maps $J \mapsto J^+(\phi)$ and $J \mapsto J^-(\phi)$ are the `reflections' mentioned in the introduction. In what follows, we will sometimes abuse notation slightly by writing $J(\phi^+)$ for $J^+(\phi)$ and $J(\phi^-)$ for $J^-(\phi)$. 

Let $J_0,J_1,J_2,\dots$ be a sequence of edge subsets of $H$. We say that it is a \emph{folding sequence} in $H$ if, for each $i\geq 1$, $J_i=J_{i-1}(\psi_i)$ 
for some half-folding map $\psi_{i}$, calling the corresponding sequence of half-folding maps $(\psi_i)_{i=1}^N$ the \emph{signature} of the folding sequence.
If a finite folding sequence $J_0,J_1,\dots,J_N$ in a graph $H$ 
starts from a set $J_0$ consisting of a single edge and ends with $J_N=E(H)$,
then we call it a \emph{percolating sequence}.
The reason for these definitions is that such folding sequences allow us to keep track of repeated applications of the Cauchy--Schwarz inequality. Indeed, a simple application of the Cauchy--Schwarz inequality gives
\begin{align}\label{eq:CS}
    t_{J_i}(W)\leq t_{J_i^+(\phi)}(W)^{1/2}t_{J_i^-(\phi)}(W)^{1/2},
\end{align}
where here we identify the edge sets $J_i$, $J_i^+(\phi)$ and $J_i^-(\phi)$ with their corresponding subgraphs. 
With this terminology in place, we may recall one of the main results in~\cite{CL17}, which says that if there is a percolating sequence in  a graph $H$, then, through appropriate repeated applications of~\eqref{eq:CS}, we have that $H$ is weakly norming. 

\begin{theorem}[Conlon--Lee] \label{thm:colouring_graphs}
	If there exists a percolating sequence $J_0,J_1,\dots,J_N$, then $H$ is weakly norming.
\end{theorem}

The goal of this section is to prove a multilayered variant of this result whose conclusion is that a particular graph $H$ is dominating.
Suppose that $H_1,H_2,\dots,H_k$ are edge-disjoint subgraphs of $H$ that partition the edge set $E(H)$. We say that $H_1,H_2,\dots,H_k$ are \emph{layers} of the graph $H$ if there is a nontrivial subset $\Phi$ of the set of cut involutions of $H$ such that $\phi(e) \in E(H_i)$ for all $\phi \in \Phi$, $i \in [k]$ and $e \in E(H_i)$. 
A \emph{layered percolating sequence} for $H$ (with respect to $H_1,H_2,\dots,H_k$, though we typically omit this) is then a folding sequence $J_0,J_1,\dots,J_N$ with signature $(\psi_i)_{i=1}^N$ such that $J_0$ consists of $k$ edges $e_1,e_2,\dots,e_k$, where $e_i\in E(H_i)$ for each $i\in [k]$,  $J_N=E(H)$ and each $\psi_i$ is a half-folding map associated to a cut involution $\phi_i \in \Phi$. The $k$ edges $e_1,e_2,\dots,e_k$ that form $J_0$ are called the \emph{seeds} of the layered percolating sequence. In particular, if there is a layered percolating sequence for $H$, then each $H_i$ is edge-transitive under the group of automorphisms of $H$
generated by the set of layer-preserving cut involutions $\Phi$.

\begin{example}\label{ex:percolate_6wheel}
On the left side of Figure~\ref{fig:base2}, the two thick edges represent one possible choice for~$J_0$ in a layered percolating sequence for $C_6^+$.
More precisely, the layers of the partition are the outer cycle~$C_6$ and the star~$K_{1,3}$ of edges adjacent to the central vertex. The inner edge percolates to cover the star~$K_{1,3}$, while the outer edge percolates to cover the cycle~$C_6$. Moreover, as required for a layered percolating sequence, they percolate together to cover the entire graph~$J$.
In contrast, the choice of thick edges on the right side of Figure~\ref{fig:base2} does \emph{not} produce a layered percolating sequence. To see this, note that any cut involution that places the two edges in opposite halves will delete one of the two edges when we reflect.
\end{example}

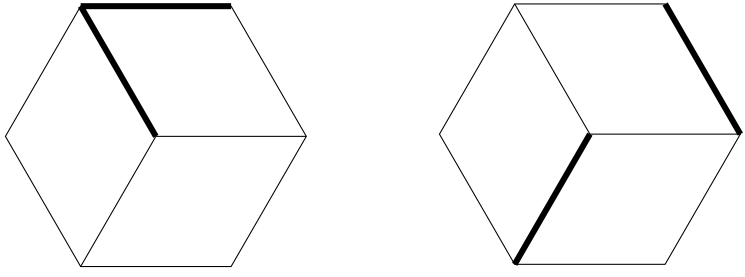
\begin{figure}
\centering
\begin{tikzpicture}
        \draw (0:\R) \foreach \x in {60,120,...,359} {
                -- (\x:\R)
            }-- cycle (90:\R) node[above] {} ;
        \draw (0:0)--(0:\R);
        \draw[line width = 0.8mm] (0:0)--(120:\R);     
        \draw (0:0)--(240:\R);   
        \draw[line width = 0.8mm] (120:\R)--(60:\R);      
\end{tikzpicture}
\hspace{15mm}
\begin{tikzpicture}
        \draw (0:\R) \foreach \x in {60,120,...,359} {
                -- (\x:\R)
            }-- cycle (90:\R) node[above] {} ;
        \draw (0:0)--(0:\R);
        \draw (0:0)--(120:\R);     
        \draw[line width = 0.8mm] (0:0)--(240:\R);   
        \draw[line width = 0.8mm] (0:\R)--(60:\R);
         
\end{tikzpicture}
\caption{Different choices for $J_0$ in $C_6^+$.} \label{fig:base2}
\end{figure}

For $I\subseteq [k]$, let $H_I = \cup_{i \in I} H_i$ be the subgraph obtained by taking the union of all the edges in the layers $H_i$ with $i\in I$. 
As the second case in Example~\ref{ex:percolate_6wheel} above shows, even if a layered percolating sequence exists, there may also be $k$-edge subsets $F$ spanning all the layers such that following the percolating sequence starting with $F$ does not lead to the whole set $E(H)$. This is a potential problem in generalising Theorem~\ref{thm:colouring_graphs}. However, we now show that, even if this happens, the folding sequence does always end with $E(H_I)$ for some $I\subseteq [k]$.

\begin{lemma}\label{lem:multiperc}
Let $H$ be a graph with layers $H_1, \dots, H_k$ and let $J_0, J_1,\dots,J_N$ be a layered percolating sequence of $H$ with  signature $(\psi_i)_{i=1}^N$. Then every folding sequence $F_0,F_1,\dots,F_N$ with the same signature satisfies $F_N=E(H_I)$ with $I= J_0\cap F_0$.
\end{lemma}

\begin{proof}
Let $F_{i,j}:=F_i\cap E(H_j)$.
Then $F_{i,j}=F_{i-1,j}(\psi_i)$, by considering the restriction of the signature $(\psi_i)_{i=1}^N$ to $E(H_j)$.
Suppose that there is $e_j\in J_0\cap F_0\cap E(H_j)$, i.e., $F_{0,j}\supseteq J_0\cap E(H_j)=\{e_j\}$. Then, by induction, 
\begin{align*}
    F_{i,j}=F_{i-1,j}(\psi_i)\supseteq \big(J_{i-1}\cap E(H_{j})\big)(\psi_i)=J_{i-1}(\psi_i)\cap E(H_j) = J_i \cap E(H_j)
\end{align*}
for each $i=1,2,\dots,N$. In particular,
$F_{N,j}=E(H_j)$.

Suppose now that $J_0\cap F_{0,j}$ is empty.
Let $\overbar{F}=E(H_j)\setminus F$ for $F\subseteq E(H_j)$.
It then follows that 
\begin{align*}
   \overbar{F}(\psi) = \{e\in E(H_j):\psi(e)\notin F\} = E(H_j)\setminus F(\psi) = \overbar{F(\psi)}.
\end{align*}
In particular,
$\overbar{F_{i,j}}=\overbar{F_{i-1,j}(\psi)}=\overbar{F_{i-1,j}}(\psi)$ and, therefore, $\overbar{F_{0,j}},\overbar{F_{1,j}},\dots,\overbar{F_{N,j}}$ is again a folding sequence with signature $(\psi_i)_{i=1}^N$.
Thus, by repeating the argument above, $\overbar{F_{N,j}}=E(H_j)$ since $J_0\cap \overbar{F_{0,j}}$ is nonempty. That is, $F_{N,j}$ is empty.

To summarise, $F_{N,j}$ is either the whole of $E(H_j)$ or empty, depending on whether $J_0\cap F_{0,j}$ is nonempty or empty. Thus, $F_N=E(H_I)$ with $I=J_0\cap F_0$.
\end{proof}

At first glance, this lemma seems to be something of a dead end. Starting from a given subgraph $F_0$, we set out to cover the entire graph, but instead only managed to cover some of its layers. It is at this point that we need to introduce our second operation. Suppose that we have a layered percolating sequence in $H$ whose set of seeds is $K$. Given a nonempty $I\subsetneq [k]$, a subgraph $H'$ of $H$ is said to be a \emph{relocation} of $H_I$ if it dominates $H_I$ and contains more than $|I|$ seeds from $K$. We then say that the \emph{layers can be relocated} if, for any nonempty $I \subsetneq [k]$, there is a relocation $H'$ of $H_I$.
In particular, if there is always a subgraph isomorphic to $H_I$ in $H$ that contains more than $|I|$ seeds, then the layers can be relocated. The rough idea is that after each relocation step we can apply the folding sequence to the relocation to fill more layers of $H$, which can in turn be relocated.

\begin{example} \label{ex:J}
As shown in Figure~\ref{fig:base3}, the layers of $C_6^+$, in this case a copy of $K_{1,3}$ and a copy of $C_6$, can both be relocated to include the set $J_0$ on the left side of Figure~\ref{fig:base2}.
\end{example}

\begin{figure}
\centering
\begin{tikzpicture}[level 2/.style={sibling distance=6.5mm}]

        \draw (0:\R) \foreach \x in {60,120,...,359} {
                -- (\x:\R)
            }-- cycle (90:\R) node[above] {} ;
        \draw (0:0)--(0:\R);
        \draw[line width = 0.8mm] (0:0)--(120:\R);     
        \draw (0:0)--(240:\R);   
        \draw[line width = 0.8mm] (120:\R)--(60:\R);      
        \draw[line width = 0.8mm] (120:\R)--(180:\R);
        
\end{tikzpicture}\hspace{15mm}
\begin{tikzpicture}[level/.style={sibling distance=30mm/#1}, level distance=10mm]    
        
        \draw (0:\R) \foreach \x in {60,120,...,359} {
                -- (\x:\R)
            }-- cycle (90:\R) node[above] {} ;
        \draw (0:0)--(0:\R);
        \draw[line width = 0.8mm] (0:0)--(120:\R);     
        \draw[line width = 0.8mm] (0:0)--(240:\R);   
        \draw[line width = 0.8mm] (120:\R)--(60:\R);      
        \draw[line width = 0.8mm] (0:\R)--(60:\R);  
        \draw[line width = 0.8mm] (300:\R)--(359:\R);  
        \draw[line width = 0.8mm] (300:\R)--(240:\R);  
\end{tikzpicture}
\caption{Relocations of the inner star $K_{1,3}$ and the outer cycle $C_6$ in $C_6^+$.} \label{fig:base3}
\end{figure}
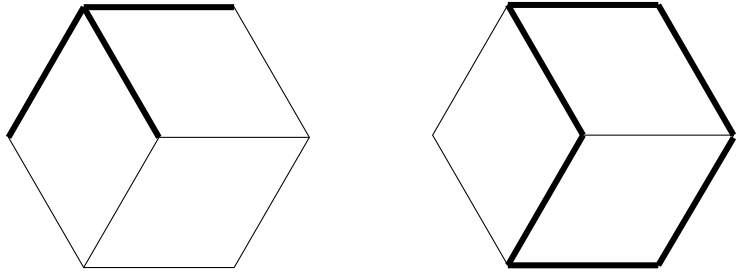

The main result of this section is now the following.

\begin{theorem} \label{thm:perc}
Suppose that $H$ is a graph with layers $H_1, \dots, H_k$. If there exists a layered percolating sequence for $H$ and the layers can be relocated, then $H$ is dominating.
\end{theorem}

\begin{proof}
The case $k=1$ is a corollary of Theorem~\ref{thm:colouring_graphs}, so we may assume $k>1$. Let $F$ be an arbitrary nonempty edge subset of $H$ and let $\Psi:=(\psi_i)_{i=1}^N$ be the signature of a layered percolating sequence $J_0, J_1,\dots,J_N$. Note that, since each layer is edge transitive, we may assume $F\cap J_0$ is nonempty.

We now define a rooted tree~$\mathcal{T}(F;\Psi)$ of depth $N+1$ which encodes which graphs are obtained through iterations of~\eqref{eq:CS} and relocations. The vertices of $\mathcal{T}(F;\Psi)$ are labelled by edge subsets of $H$ with the root labelled by the initial edge subset $F$. Each vertex at depth $d \le N$, labelled with~$J$, say, has two children
with labels $J(\psi_d)$ and $J(\overbar{\psi_d})$. If, at depth $N$, a vertex has label $J=E(H_I)$ for some nonempty $I\subsetneq [k]$, then it has a unique child labelled by its relocation $J'$. Otherwise, it has a unique child labelled by the same $J$ as the parent.
We now abuse notation slightly by using $J$ to denote both an edge subset of $H$ and the corresponding subgraph. 
With this convention, each parent $J$ relates to its children $J(\psi)$ and $J(\overbar{\psi})$ or its only child $J'$ through the inequalities
\begin{align*}
    t_{J}(W)\leq t_{J(\psi)}(W)^{1/2}t_{J(\overbar{\psi})}(W)^{1/2}~~\text{ or }~~ t_{J}(W)\leq t_{J'}(W)^{e(J)/e(J')}.
\end{align*}
In either case, the sum of the size of the edge set times the corresponding exponent is the same on  both sides, i.e., $|J|=\tfrac{1}{2}|J(\psi)|+\tfrac{1}{2}|J(\overbar{\psi})|$ and $|J|=|J'|\cdot |J|/|J'|$. This is therefore an invariant throughout the entire process.

Denote by $F_{1},\dots,F_{2^d}$ the edge subsets that label the vertices at depth $d\le N$. 
Because we applied inequality~\eqref{eq:CS} iteratively at each depth,
we have the bound
\begin{align*}
    t_{F}(W) \leq \prod_{i=1}^{2^d} t_{F_{i}}(W)^{1/2^d},
\end{align*}
where $\sum_{j=1}^{2^d}2^{-d}|F_{j}|=|F|$. 
When $d=N$, by Lemma~\ref{lem:multiperc}, each $F_{j}$ is equal to $E(H_{I_j})$ for some $I_j\subseteq [k]$ and, moreover, at least one of the $I_j$ is nonempty, since $J_0\cap F$ was nonempty. Then the single child of any $F_j$ with $I_j\neq\emptyset$ is the relocation $F_j'$ of $F_j$, which, unless $I_j=[k]$,  contains more than $|I_j|$ seeds. If $I_j$ is either empty or $[k]$, then no relocation is possible. Thus, at depth $N+1$, the corresponding inequality is
\begin{align}\label{eq:first_iteration}
    t_{F}(W) \leq \prod_{i=1}^{2^N} t_{F_{i}'}(W)^{\alpha_i},
\end{align}
where $F_i'$ is the label of the unique child of $F_i$, $\alpha_i = |F_i|/|F'_i|2^N$ and $\sum_{i=1}^{2^N} \alpha_i |F'_i| = |F|$.

We now expand this tree by adding the tree $\mathcal{T}(F_j';\Psi)$ starting from each leaf. 
Once again, one of the leaves is equal to a relocation of some $E(H_J)$, which, unless $J=[k]$, contains more than $|J|$ seeds. Thus, for each $j=1,2,\dots, 2^N$,
\begin{align*}
    t_{F_j'}(W) \leq \prod_{i=1}^{2^N} t_{F_{i,j}^{(2)}}(W)^{\alpha_{i,j}},
\end{align*}
where $\sum_{i=1}^{2^N} \alpha_{i,j} |F_{i,j}^{(2)}| = |F'_j|$. Hence, combined with~\eqref{eq:first_iteration}, we have that 
\begin{align*}
    t_{F}(W) \leq \prod_{i=1}^{2^{2N}} t_{F_{i}^{(2)}}(W)^{\beta_i},
\end{align*}
where $F_{i}^{(2)}$ is a reindexing of the $F_{i,j}^{(2)}$'s and $\sum_{i=1}^{2^{2N}} \beta_i |F_{i}^{(2)}| = |F|$. We continue to iterate this process, noting that, each time we iterate, we increase the maximum number of seeds in the graphs that appear in the corresponding upper bound for $F$, until it reaches $k$. 
Therefore, if we iterate $k-3$ more times, we must obtain a bound of the form 
\begin{align}\label{eq:firstbound}
    t_{F}(W) \leq \prod_{i=1}^{N_k} t_{K_{i}}(W)^{\rho_i},
\end{align}
where $N_k=2^{(k-1)N}$, $\sum_{i=1}^{N_k} \rho_i|K_{i}|=|F|$ and at least one of the $K_j$, say $K_1$, is equal to $E(H)$. Since the choice of $F$ was arbitrary, each $K_j$ also satisfies
\begin{align*}
    t_{K_j}(W) \leq \prod_{i=1}^{N_k} t_{K_{i}^{(j)}}(W)^{\rho_{i,j}}
\end{align*} 
for some $\rho_{i,j}$ with $\sum_{i=1}^{N_k} \rho_{i,j} |K_{i}^{(j)}| = |K_j|$. 
Note that if $K_j$ equals either $\emptyset$ or $E(H)$, then so do all of its descendants $K_{i}^{(j)}$, while if $K_j$ is neither $\emptyset$ nor $E(H)$, then at least one of its descendants, say $K_{1}^{(j)}$, is equal to $E(H)$.
Substituting this back into~\eqref{eq:firstbound} yields, after a suitable relabelling, that
\begin{align*}
    t_{F}(W) \leq \prod_{i=1}^{N_k^2} t_{K_{i}}(W)^{\sigma_i},
\end{align*}
where $\sum_{i=1}^{N_k^2} \sigma_i |K_i| = |F|$ and the proportion of the $K_j$ which are equal to neither $\emptyset$ nor $E(H)$ has dropped to at most a $(1 - 1/N_k)$-factor of what it was. Repeating this process, we see that the proportion of $K_j$ which are equal to neither $\emptyset$ nor $E(H)$ converges to $0$ and the proportion which are equal to $E(H)$ converges to some limit $\gamma$. But then we must have that $\gamma e(H) = e(F)$ and $t_F(W) \leq t_H(W)^\gamma$, which together yield the required domination inequality.
\end{proof}

When applying Theorem~\ref{thm:perc}, there are two conditions to check: the existence of layered percolating sequences and the possibility of relocations. While verifying the second condition can be a little ad hoc, there is an approach to finding graphs with layered percolating sequences that builds on our earlier work~\cite{CL17} connecting weakly norming graphs and reflection groups. We will be rather terse here, but we refer the reader to~\cite{CL17} for a more comprehensive introduction to reflection groups.

Let $\W$ be a finite reflection group with $S$ a set of simple reflections and let $\W_1,\W_2,\dots,\W_k$ be subgroups of $\W$ generated by subsets $S_1,S_2,\dots,S_k$ of $S$, respectively. Then the $(S_1,\dots,S_k;S,\W)$-reflection hypergraph, shown in~\cite{CL17} to be weakly norming, is the $k$-partite $k$-uniform hypergraph whose parts are the cosets of $\W_i$ for each $i = 1, \dots, k$, with an edge for every $k$-tuple of the form $(w\W_1,w\W_2,\dots,w\W_k)$ with $w \in \W$. If we now replace each hyperedge $(w\W_1,w\W_2,\dots,w\W_k)$ by a $(k-1)$-star centred at $w\W_i$, we obtain a bipartite graph\footnote{If multiple edges occur, then we simplify them.}~between the cosets of $\W_i$ and the union of the cosets of $\W_1,\dots, \W_{i-1}, \W_{i+1}, \dots,\W_k$, which we call the \emph{$(S_1,\dots,S_k;i,S,\W)$-graph}. 
Since the induced subgraph of this graph between the cosets of $\W_i$ and $\W_j$, for each $j\neq i$, is the $(S_i,S_j;S,\W)$-reflection graph and the set of reflections $t\in\W$ give cut involutions of these graphs
(see~\cite[Corollary 4.9]{CL17}),  
the graph naturally has layers.

\begin{example}
Let $\W$ be the symmetric group on the three elements $\{1,2,3\}$ and let $S=\{\sigma_{12},\sigma_{23}\}$, where $\sigma_{ij}$ is the permutation that swaps $i$ and $j$. Then the graph $C_6^+$ in Figure~\ref{fig:base} is isomorphic to the $(S_1,S_2,S_3;1,S,\W)$-graph with  $S_1=\{\sigma_{12}\}$, $S_2=\{\sigma_{23}\}$ and $S_3=S$.
\end{example}

\begin{theorem}\label{thm:multi_perc}
Let $H$ be the $(S_1,S_2,\dots,S_k;i,S,\W)$-graph. Then there exists a layered percolating sequence $J_0,J_1,\dots, J_N$ such that $J_0$ is the edge set of the $(k-1)$-star induced on $\{\W_1,\dots,\W_k\}$.
\end{theorem}

We omit the proof, which is virtually identical to that of~\cite[Theorem~1.2]{CL17}.

\section{Constructions of dominating graphs} \label{sec:examples}

We now construct various examples of dominating graphs. 
Our first such family is as follows and already includes our running example $J = C_6^+$ from Figure~\ref{fig:base}.

\begin{theorem} \label{thm:star}
Suppose that $H$ is a reflection graph with one side $A$ in its bipartition of order $a$ and regular of degree $d$. If $d\geq a-1$, then the graph $H_A^+$ formed by joining a new vertex to each vertex in $A$ is dominating.
\end{theorem}

\begin{proof}
Suppose that $H$ is the $(S_1,S_2;S,\W)$-reflection graph, where $A$ is the family of all cosets of the form $w\W_1$. Let $S_3=S$, so that $\W_3=\W$. Then $H_A^+$ is the $(S_1,S_2,S_3;1,S,\W)$-graph and, hence, by Theorem~\ref{thm:multi_perc}, there exists a layered percolating sequence starting with two edges that form a star centred at $v\in A$. In fact, \emph{any} such  $2$-star does the job, because of the edge-transitivity of each layer. 

By Theorem~\ref{thm:perc}, it remains to show that the layers can be relocated. Let $v_0$ be the vertex adjacent to all the vertices in $A$ and let $B$ be the other side of the bipartition of $H$ to $A$.
First observe that the layer isomorphic to the $(S_1,S_2;S,\W)$-graph $H$ can be relocated by using the vertex $v_0$. Concretely, we may replace any $b\in B$ and its incident edges by $v_0$ and the edges between $v_0$ and the neighbours of $b$, since $v_0$ is adjacent to all vertices in $A$. Provided $d \geq 2$ (the easy case where $d = 1$ can 
 be dealt with separately), this relocation contains two seeds that form a $2$-star centred on $A$, whereas $H$ only contained one seed.
The other layer between $v_0$ and $A$ is just a copy of an $a$-star centred at $v_0$. This star can easily be relocated to another $a$-star centred in $A$ containing more seeds than the star centred at~$v_0$, where we use that each $v\in A$ has degree $d+1\geq a$ in~$H_A^+$.
\end{proof}

This theorem adds many more examples of dominating graphs besides $C_6^+$. For instance, Figure~\ref{fig:H+} shows the dominating graphs obtained by adding a vertex to four vertex-disjoint copies of $K_{1,4}$, the 1-subdivision of $K_4$ and the 3-dimensional cube. 

\R = 1.2cm

\newdimen\S
\S=2.4cm

\newdimen\W
\W = 1.4cm

\newdimen\T
\T = 1.6cm

\newdimen\U
\U = 1cm

\begin{figure}
\centering

		\begin{tikzpicture}[level distance=1.5cm,
  level 1/.style={sibling distance=1.2cm},
  level 2/.style={sibling distance=0.3cm},
  common/.style={circle,fill,inner sep =0pt, minimum size=4pt}]
  \node [common]{}
    child {node [common]{}
      child{ node [common]{} }
      child {node [common]{} }
      child {node [common]{} }
      child {node [common]{} }
    }
    child {node [common]{}
      child{ node [common]{} }
      child {node [common]{} }
      child {node [common]{} }
      child {node [common]{} }
    }
    child {node [common]{}
      child{ node [common]{} }
      child {node [common]{} }
      child {node [common]{} }
      child {node [common]{} }
    }
    child {node [common]{}
      child{ node [common]{} }
      child {node [common]{} }
      child {node [common]{} }
      child {node [common]{} }
    };
		\end{tikzpicture}\hspace{15mm}
\begin{tikzpicture}[level 2/.style={sibling distance=6.5mm}]

        \draw (45:0)--(45:\W);
        \draw (135:0)--(135:\W);     
        \draw (225:0)--(225:\W);          
        \draw (315:0)--(315:\W);
        
        \draw (45:\W)--(0:\T);
        \draw (315:\W)--(0:\T); 
            
        \draw (45:\W)--(90:\T);          
        \draw (135:\W)--(90:\T);
        
        \draw (135:\W)--(180:\T);
        \draw (225:\W)--(180:\T); 
            
        \draw (225:\W)--(270:\T);          
        \draw (315:\W)--(270:\T);

        \draw (0:\U)--(45:\W);
        \draw (0:\U)--(225:\W);
        
        \draw (180:\U)--(135:\W);
        \draw (180:\U)--(315:\W);           
        
        \node at (45:\W) [circle,fill,inner sep =0pt, minimum size=4pt] () {};
        \node at (0:\T) [circle,fill,inner sep =0pt, minimum size=4pt] () {};
        
        \node at (135:\W) [circle,fill,inner sep =0pt, minimum size=4pt] () {};
        \node at (90:\T) [circle,fill,inner sep =0pt, minimum size=4pt] () {};
        
        \node at (225:\W) [circle,fill,inner sep =0pt, minimum size=4pt] () {};
        \node at (180:\T) [circle,fill,inner sep =0pt, minimum size=4pt] () {};
        
        \node at (315:\W) [circle,fill,inner sep =0pt, minimum size=4pt] () {};
        \node at (270:\T) [circle,fill,inner sep =0pt, minimum size=4pt] () {};

        \node at (0:\U) [circle,fill,inner sep =0pt, minimum size=4pt] () {};
        \node at (180:\U) [circle,fill,inner sep =0pt, minimum size=4pt] () {};
       
        \node at (0:0) [circle,fill,inner sep =0pt, minimum size=4pt] () {};        
        
        \node at (315:\S) () {};

\end{tikzpicture}\hspace{15mm}
\begin{tikzpicture}[level/.style={sibling distance=30mm/#1}, level distance=10mm]    

        \draw (45:0)--(45:\R);
        \draw (135:0)--(135:\R);     
        \draw (225:0)--(225:\R);          
        \draw (315:0)--(315:\R);
        
        \draw (45:\R)--(45:\S);
        \draw (45:\R)--(135:\S);
        \draw (45:\R)--(315:\S);
        
        \draw (135:\R)--(45:\S);
        \draw (135:\R)--(135:\S);
        \draw (135:\R)--(225:\S);
        
        \draw (225:\R)--(225:\S);
        \draw (225:\R)--(135:\S);
        \draw (225:\R)--(315:\S);
        
       \draw (315:\R)--(45:\S);
        \draw (315:\R)--(225:\S);
        \draw (315:\R)--(315:\S);
        
        \node at (45:\R) [circle,fill,inner sep =0pt, minimum size=4pt] () {};
        \node at (45:\S) [circle,fill,inner sep =0pt, minimum size=4pt] () {};
        
        \node at (135:\R) [circle,fill,inner sep =0pt, minimum size=4pt] () {};
        \node at (135:\S) [circle,fill,inner sep =0pt, minimum size=4pt] () {};
        
        \node at (225:\R) [circle,fill,inner sep =0pt, minimum size=4pt] () {};
        \node at (225:\S) [circle,fill,inner sep =0pt, minimum size=4pt] () {};
        
        \node at (315:\R) [circle,fill,inner sep =0pt, minimum size=4pt] () {};
        \node at (315:\S) [circle,fill,inner sep =0pt, minimum size=4pt] () {};
        
        \node at (0:0) [circle,fill,inner sep =0pt, minimum size=4pt] () {};
\end{tikzpicture}
\caption{Examples of $H_A^+$.} \label{fig:H+}
\end{figure}
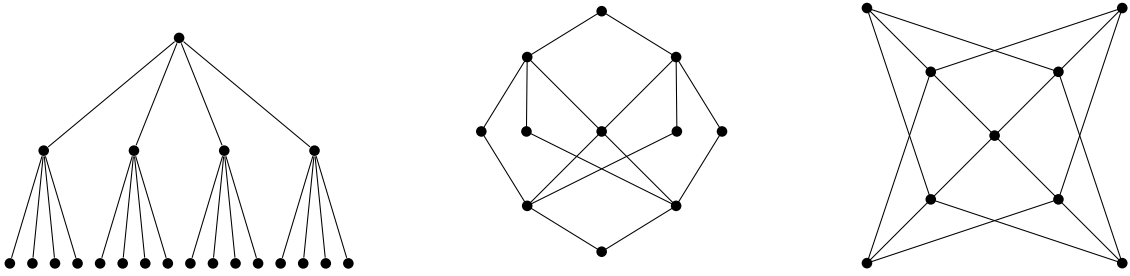

\medskip

For a graph $H$, the \emph{$K_{2,t}$-replacement} of $H$ is the graph obtained by replacing each edge with a copy of $K_{2,t}$, identifying the edge's endpoints with the two vertices on one side of the
corresponding copy of $K_{2,t}$. In particular, the $K_{2,1}$-replacement of $H$ is just the $1$-subdivision of $H$. Alternatively, the $K_{2,t}$-replacement can be viewed as replacing each edge of $H$ by $t$ multiedges and then $1$-subdividing this multigraph. When we form the $K_{2,t}$-replacement of a graph, the side of the bipartition that corresponds to the original set of vertices is typically smaller. Thus, by Proposition~\ref{prop:regular}, for this graph to be dominating, the original graph $H$ must be regular. We now prove a sort of converse to this observation, showing that if $H$ is a regular reflection graph, then the $K_{2,t}$-replacement of $H$ is dominating.

\begin{theorem}\label{thm:K2t-replace}
The $K_{2,t}$-replacement of a connected regular reflection graph $H$ is dominating.
\end{theorem}
\begin{proof}
Suppose that $A\cup B$ is the bipartition of $H$.
Let $H'$ be the $K_{2,t}$-replacement of $H$ and let $A'\cup B'$ be its bipartition, where $A'=V(H)$ and $B'$ is the disjoint union of $t$ copies of $E(H)$. An edge in $H'$ then represents an incidence between the corresponding vertex and edge in $H$.

For a vertex $b\in B'$, let $E_b$ be the set of vertices in $B'$ that have the same two neighbours as $b$, i.e., the set of all copies of the same edge in $H$ as $b$.
For each $b'\in E_b$, 
there exists a cut involution $\phi_{b,b'}$ of $H'$ that fixes all other vertices, but maps $b$ and $b'$ to each other. Moreover, the cut involutions $\phi$ of $H$ extend naturally to cut involutions $\phi'$ of $H'$ where $\phi'(v)=\phi(v)$ for $v\in V(H)=A'$ and $\phi'(u)$ is the $j$-th copy of $\phi(e)$ if $u\in B'$ is the $j$-th copy of $e\in E(H)$.
Together, these cut involutions easily show that the two isomorphic induced subgraphs on $A\cup B'$ and $B\cup B'$ are layers in $H'$.

We claim that any two edges $ab$ and $bc$ with $a \in A$, $b \in B'$ and $c \in B$ 
give an appropriate $J_0$ for a layered percolating sequence. 
Indeed, by using the cut involutions $\phi_{b,b'}$ for every pair $(b,b')$, where $b'$ is a copy of the same $H$-edge as $b$, one obtains a folding sequence $J_0,J_1,\dots,J_{t-1}$, where $J_{t-1}$ contains all edges of the forms $ab'$ and $b'c$ with $b' \in E_b$. If we then use the standard percolating sequence for $H$ given by~\cite[Theorem~4.12]{CL17}, we obtain a layered percolating sequence.

By Theorem~\ref{thm:multi_perc}, it remains to check that the layers can be relocated. Recall that the two layers are the induced subgraphs on $A\cup B'$ and $B\cup B'$, respectively, which are  isomorphic copies of $|A|$ disjoint $td$-stars, where $d$ is the degree of a vertex in the regular graph $H$. To relocate the layer on $A\cup B'$, we replace a vertex $a\in A$ by a vertex $b\in B$. Then the $td$-star centred at $a$ is replaced by another $td$-star that shares $t$ leaves with each of the $d$ different $td$-stars centred at the neighbours of $b$ in $H$.
Let $J$ be this edge-disjoint (but \emph{not} vertex-disjoint) union of $d+1$ different $td$-stars. Then $J$ consists of $d$ copies of $K_{2,t}$ such that the two-vertex side of each $K_{2,t}$ consists of a vertex $v\in A$ and $b$, where $v$ and $b$ are adjacent in $H$, and $t(d-1)$ pendant leaves attached to each $v\in A$. 
The induced subgraph on $A\setminus\{a\}\cup\{b\}\cup B'$ is the disjoint union of $J$ and $|A|-d-1$ further $td$-stars.  Theorem~2.7 in \cite{L19} now implies that 
\begin{align*}
    t_J(W) \geq t_{S}(W)^{d+1},
\end{align*}
where $S$ denotes the $td$-star. Thus, $J$ dominates a $td$-star and, moreover, it contains two seeds. It is then simple to conclude that the induced subgraph on $A\setminus\{a\}\cup\{b\}\cup B'$ is a relocation of the layer induced on~$A\cup B'$.
\end{proof}

In particular, this result clearly implies our Theorem~\ref{thm:examp}, since the complete bipartite graphs $K_{t,t}$ are themselves regular reflection graphs. 
For more examples, suppose $r < t$ and consider the bipartite graph 
between the set of $r$-subsets and the set of $(t-r)$-subsets of $[t]$ where two sets are connected if and only if the smaller of the two sets is contained in the larger one. This family of bipartite Kneser graphs, as they are known, includes $K_{t,t}$ with a perfect matching removed and the middle layer graph.  They are all regular reflection graphs and so, by Theorem~\ref{thm:K2t-replace}, their $1$-subdivisions, and all their $K_{2,t}$-replacements, are  dominating.

We have already seen that some trees are dominating, including even-length paths and the example shown in Figure~\ref{fig:H+}, but there are more examples.
A \emph{perfect $d$-regular tree} is a tree $T$ rooted at a vertex $r$ such that $r$ has $d$ children, every vertex other than $r$ or the leaves has $d-1$ children and every leaf is at the same distance from~$r$. In particular, every path of length $2k$ is a $2$-regular tree of depth $k$.
Thus, the following theorem generalises Godsil's result on even-length paths.

\begin{theorem}\label{thm:tree}
Every perfect $d$-regular tree $T$ is dominating.
\end{theorem}

\begin{proof}
We proceed by induction on the depth $k$ of the tree. As the statement is simple for $k = 1$, we may assume that $T$ has depth $k$ and any subgraph of the perfect $d$-regular tree of depth $k-1$ is dominated by that tree. 
Since it is clear that $T$ has layered percolating sequences, with the edges of any shortest path from the root $r$ to a leaf serving as the seeds, it will suffice to show that the layers can be relocated. The layers here are just the bipartite graphs between the vertices of depths $i$ and $i+1$ for $0 \leq i < k$, so if we have a union of at most $k-1$ layers, then each of the connected components of the resulting graph is a subgraph of the perfect $d$-regular tree of depth $k-1$, so they and their union are dominated by this tree. 
But this tree is easily seen to be isomorphic to a subgraph of $T$ that includes a shortest path from $r$ to the leaves, so the layers can indeed be relocated.
\end{proof}

To close this section, we note that, unlike weakly norming graphs (see~\cite{H10}), the family of dominating graphs is not closed under taking (bipartite) tensor products. For example, consider the tensor product $C_6^+ \times C_6^+$, where the two copies of $C_6^+$ are oriented the opposite way. Then each side of the bipartition of $C_6^+ \times C_6^+$ has vertices with two distinct degrees and so, by Lemma~\ref{prop:regular}, the graph cannot be dominating. However, as shown by the following lemma, which 
 says that blowups of dominating graphs are again dominating, some tensor products are still allowed. 

\begin{lemma}
If $H$ is a dominating graph, then so is $H \times K_{m,m}$ for any integer $m \geq 1$.
\end{lemma}

We will only sketch the proof of this lemma, working throughout in the language of graphs rather than graphons. Suppose then that $J$ is a subgraph of $H \times K_{m,m}$ and we would like to show that $t_J(G) \leq t_{H \times K_{m,m}}(G)$ for all graphs $G$. The first step is to show that there exist subgraphs $J_1, \dots, J_t$ of $H$ and positive real numbers $\alpha_1, \dots, \alpha_t$ with $\sum_{i=1}^t \alpha_i = 1$ such that $t_J(G) \leq \prod_{i=1}^t t_{J_i \times K_{m,m}}(G)^{\alpha_i}$. For this, we consider an edge $e = (u,v)$ of $H$ such that $J$ contains an edge of the form $(e, \cdot)$. Then, if we fix all vertices in $H \times K_{m,m}$ except those of the form $(u, \cdot)$ and $(v, \cdot)$ and use the fact that $K_{m,m}$ has a percolating sequence, we can show that $t_J(G) \leq \prod_{i=1}^{s} t_{B_{i}}(G)^{\beta_{i}}$, where each $B_{i}$ is a subgraph of $H \times K_{m,m}$ which includes either all or none of the edges in $e \times K_{m,m}$ and the $\beta_{i}$ are positive real numbers with $\sum_{i=1}^{s} \beta_{i} = 1$. Repeating this process for all edges $e = (u,v)$ of $H$ such that $J$ contains an edge of the form $(e, \cdot)$ implies that $t_J(G) \leq \prod_{i=1}^t t_{F_i}(G)^{\alpha_i}$, where each $F_i$ is a subgraph of $H \times K_{m,m}$ which includes either all of none of the edges in $e \times K_{m,m}$ for all $e \in E(H)$ and $\alpha_1, \dots, \alpha_t$ are positive real numbers with $\sum_{i=1}^t \alpha_i = 1$. But then each $F_i$ is of the form $J_i \times K_{m,m}$, giving exactly the inequality we require.

The second step is to show that $J_i \times K_{m,m}$ is dominated by $H \times K_{m,m}$ for each $J_i$. For this, given a graph $G$, consider the auxiliary graph $G^*$ whose vertex set $A$ is the set of (not necessarily distinct) ordered $m$-tuples in $G$ with an edge between $a$ and $a'$ if there is a homomorphic copy of $K_{m,m}$ between $a$ and $a'$ in $G$. Then, since $H$ is dominating,
\[t_{J_i \times K_{m,m}}(G) = t_{J_i}(G^*) \leq t_H(G^*)^{e(J_i)/e(H)} = t_{H \times K_{m,m}}(G)^{e(J_i \times K_{m,m})/e(H \times K_{m,m})},\]
as required. Combining the results of the two steps then completes the proof.

We note that the same proof works with $K_{m,m}$ replaced by any other vertex-transitive reflection graph, such as an even cycle or a hypercube. Moreover, a similar result also holds for $H \times K_{m,n}$ with $m \neq n$ or even with $K_{m,n}$ replaced by any other reflection graph, but one needs to assume that $H$ satisfies an appropriate bipartite version of the domination property (which is indeed satisfied by all of our examples) for the proof to go through.

\section{Applications of the domination property} \label{sec:apps}

\subsection{Sidorenko's conjecture}

Building on earlier work applying entropy techniques to Sidorenko's conjecture~\cite{CKLL, KLL16, LiSz, Sz14}, it was shown in~\cite[Section 5]{CL17} that weakly norming graphs can be used as building blocks to produce a family of graphs satisfying the conjecture. The proof of that result only relied on the fact that weakly norming graphs are dominating, so it easily extends to dominating graphs. To say more, we need some definitions. 

A \emph{tree decomposition} of a graph $H$ is a pair $(\mathcal{F}, \TT)$ consisting of a family $\mathcal{F}$  of vertex subsets of $H$ and  a tree $\TT$ on $\mathcal{F}$ such that
\begin{enumerate}
\item $\bigcup_{X\in\mathcal{F}}X=V(H)$, 
\item for each $e \in E(H)$, there is a set $X \in \mathcal{F}$ such that
$X$ contains $e$,
\item for $X,Y,Z\in \mathcal{F}$, $X\cap Y\subseteq Z$ 
whenever $Z$ lies on the path from $X$ to $Y$ in $\TT$.
\end{enumerate}

Given a graph $H$ and an induced subgraph $J$,
a \emph{$J$-decomposition} of a graph $H$ is a tree decomposition $(\mathcal{F},\TT)$ of $H$ 
satisfying the following two extra conditions:
\begin{enumerate}
\item 
each induced subgraph $H[X]$, $X \in \FF$, is isomorphic to $J$,
\item 
for every pair $X,Y\in \FF$ which are adjacent in $\TT$, there is an isomorphism between the two copies $H[X]$ and $H[Y]$
of $J$ that fixes $X \cap Y$.
\end{enumerate}
We say that a graph is \emph{$J$-decomposable} if it admits a $J$-decomposition. Our application of the domination property to Sidorenko's conjecture is now as follows.

\begin{theorem}
If $J$ is a dominating graph, then every $J$-decomposable graph $H$ satisfies Sidorenko's conjecture. That is, for any graphon $W$, $t_H(W)\geq t_{K_2}(W)^{e(H)}$.
\end{theorem}

As in~\cite{CL17}, we expect that entropy techniques can be used to broaden this class further, but we have chosen not to pursue this direction here.

\subsection{The Erd\H{o}s--Simonovits conjecture}

The supersaturation conjecture of Erd\H{o}s and Simonovits~\cite{S84} is often cited as an equivalent formulation of Sidorenko's conjecture, but it is in fact stronger 
for sparse graphs. Recall that $\mathrm{ex}(n,H)$ is the largest number of edges in an $H$-free graph with $n$ vertices. Their conjecture is then as follows.

\begin{conjecture}[Erd\H{o}s--Simonovits]
Let $H$ be a bipartite graph. 
Then there exist positive constants $c$ and $C$ such that every $n$-vertex graph $G$ with at least $C\cdot \mathrm{ex}(n,H)$ edges contains at least $c\cdot n^{v(H)}p^{e(H)} $ copies of $H$, where $p=t_{K_2}(G)$.
\end{conjecture}

In full generality, this conjecture is known for only a handful of special cases. However, here we show that dominating graphs do satisfy the conjecture to some nontrivial extent. In the proof below, we will use the notation $\Hom(H,G)$ to denote the set of homomorphisms from a graph $H$ to another graph $G$. In particular, if $G$ has $n$ vertices, $t_H(G) = |\Hom(H,G)|/n^{v(H)}$. 

\begin{theorem}\label{thm:supersat}
Let $H$ be a bipartite graph with maximum degree $\Delta$. If $H$ dominates $H\setminus v$ for each $v\in V(H)$ and each $H\setminus v$ satisfies Sidorenko's conjecture,
then there exist positive constants $c$ and $C$ such that every $n$-vertex graph with at least $Cn^{2-1/\Delta}$ edges contains at least $cn^{v(H)}p^{e(H)}$ copies of $H$, where $p=t_{K_2}(G)$.
\end{theorem}

\begin{proof}
Suppose, for the sake of contradiction, that there are no such constants $c,C>0$. That is, given any $c,C>0$, there exists an $n$-vertex graph $G$ with $e(G)\geq Cn^{2-1/\Delta}$ such that at least a $(1-c)$-proportion of the homomorphic copies of $H$ in $G$ are degenerate.
As each degenerate copy of $H$ is contained in a homomorphic copy of $H\setminus v$ for some $v\in V(H)$, it follows that
\begin{align*}
    \sum_{v\in V(H)}|\Hom(H\setminus v,G)|\geq (1-c)|\Hom(H,G)|.
\end{align*}
Let $H'$ be the subgraph of $H$ that attains the maximum value of $|\Hom(H\setminus v,G)|$ amongst all $H\setminus v$ for $v\in V(H)$. In particular, $v(H)|\Hom(H',G)|\geq (1-c)|\Hom(H,G)|$.
Dividing both sides by $n^{v(H)}$ then gives $\frac{v(H)}{n}t_{H'}(G)\geq (1-c)t_{H}(G)$. Combining this with the domination inequality $t_{H}(G)\geq t_{H'}(G)^{e(H)/e(H')}$ and the Sidorenko inequality $t_{H'}(G)\geq t_{K_2}(G)^{e(H')}$ gives 
\begin{align*}
    \frac{v(H)}{(1-c)n} \geq t_{H'}(G)^{e(H)/e(H')-1} \geq t_{K_2}(G)^{e(H)-e(H')}\geq p^{\Delta},
\end{align*}
where the final inequality used $e(H)-e(H')\leq \Delta$.
Therefore, $p\leq \left(\frac{v(H)}{1-c}\right)^{1/\Delta} n^{-1/\Delta}$. Now choosing $c=1/2$ and $C=3v(H)$ yields the desired contradiction.
\end{proof}

\section{Concluding remarks} \label{sec:conc}

{\bf Which graphs are dominating?} Following~\cite{CL17}, it is reasonable to conjecture that a graph is weakly norming if and only if it has a percolating sequence, but we suspect that there is no such clean characterisation for dominating graphs. However, there are several interesting questions whose answers would help to clarify the nature of this family.

For instance, is it the case that every dominating graph arises through our reflection/relocation mechanism? And, if so, is it the case that they can all be associated in some nontrivial way to a reflection group? It is a little hard to make these questions more precise, but perhaps there is some example we missed that answers both in the negative.

We have already seen that there are dominating graphs, such as our running example $C_6^+$, which are not edge transitive. This again sets them apart from weakly norming graphs, since Sidorenko~\cite{S20} has shown that any weakly norming graph must be edge transitive. However, all of our bi-regular examples are edge transitive. 
It would be interesting to decide if this is a more general phenomenon.

It would also be interesting to decide whether there are examples other than trees 
where we need to look at three or more distinct layers (or, alternatively, 
where we need to relocate two or more times) to confirm the domination property. For both Theorem~\ref{thm:star} and~\ref{thm:K2t-replace}, we use only two layers and, despite searching extensively, we did not find any interesting examples with more layers. A related problem is to find an example of a dominating graph which has more than two distinct degrees on the larger side or, which would be more interesting, to show that no such graph exists.

Finally, as with our work on weakly norming graphs, most of the concepts and results in this paper extend to hypergraphs. We will not go into this in any depth here, but many of the same questions can also be asked in this more general context.

\vspace{3mm}
\noindent
{\bf Strong domination.} We say that a graph $H$ {\it strongly dominates} another graph $H'$ if
\[|t_H(W)|^{1/e(H)} \geq |t_{H'}(W)|^{1/e(H')}\]
for any bounded symmetric measurable function $W$ from $[0,1]^2$ to $\mathbb{R}$. That is, $W$ is no longer required to take positive values. 
This is in line with the distinction between norming and weakly norming graphs in~\cite{H10}. We then say that a graph $H$ is {\it strongly dominating} if it strongly dominates all of its subgraphs. As in~\cite{CL17}, our results here have analogues in the strongly dominating case. For instance, the analogue of Theorem~\ref{thm:perc} holds provided all of our reflections and relocations correspond to strong domination inequalities. In practice, this means that all of our cut involutions must be {\it stable}, meaning that the fixed set of the involution is an independent (or stable) set. However, besides the family of norming graphs themselves (see~\cite{LS22} for an in-depth study of which reflection graphs are norming), the only strongly dominating graphs we were able to identify are even-length paths. It would be interesting to find more.

These notions also suggest the study of which graphs have the {\it strong Sidorenko property}, that $H$ strongly dominates $K_2$, i.e., \[|t_H(W)| \geq |t_{K_2}(W)|^{e(H)}\]
for all bounded symmetric measurable functions $W:[0, 1]^2 \to \mathbb{R}$. Clearly, any such graph must be bipartite. If $H$ has the stronger property that it strongly dominates $K_{1,2}$, then $H$ must also be positive, in the sense that $t_H(W) \geq 0$ for all $W$ (see~\cite{ACetal16} for more on positive graphs, including a tantalising conjecture about their structure). To show this, suppose that $t_H(W) < 0$ and choose some $W'$ such that $t_H(W') > 0$. Then there exists $p > 0$ such that $t_H(W + pW') = 0$. However, unless $W + p W'$ averages to zero along almost every fibre, a situation which an appropriate choice of $W'$ avoids, we must have $t_{K_{1,2}}(W+pW') > 0$, contradicting the assumption that $H$ strongly dominates $K_{1,2}$. We believe that a similar conclusion should hold if we only know that $H$ strongly dominates $K_2$, namely, that once we delete any isolated edges from $H$, the remaining graph $H'$ should be positive. 

It is tempting also to conjecture a converse, saying that any positive bipartite graph has the strong Sidorenko property. Unfortunately, this is false. Indeed, the disjoint union $2H$ of two copies of $H$ is always positive, but it does not have the strong Sidorenko property unless $H$ itself does. It remains a tantalising problem to classify, or at least formulate a reasonable conjecture about, those graphs which do have the strong Sidorenko property.

\vspace{3mm}
\noindent
{\bf Other domination inequalities.} Given a graph $H$ and a vertex $x$, the graph $H_x(k)$ is the subgraph induced by all vertices of distance at most $k$ from $x$. If $H$ is vertex transitive, then we may omit the subscript $x$. We make the following conjecture.

\begin{conjecture} \label{conj:levels}
Suppose that $H$ is a vertex-transitive reflection graph with diameter $d$. Then $H(\ell)$ dominates $H(k)$ for all $1 \leq k \leq \ell \leq d$.
\end{conjecture}

Why believe this conjecture? For one thing, it holds when $H$ is an even cycle, as follows easily from the fact that even-length paths are dominating. But, less obviously, it also holds when $H$ is the $n$-dimensional cube $Q_n$. This again follows from our reflection/relocation mechanism. Indeed, if we label the vertices of the cube with the subsets of $[n]$, then, for $x$ the empty set, $H(k)$ consists of all subsets of $[n]$ with at most $k$ elements. If we relocate $H(k)$ by moving each vertex $S$ to $S \triangle \{i\}$ for some $i$ and then percolate, this increases the number of layers by one. Iterating then gives the result. 

Observe that $Q_3(2) = C_6^+$ is dominating, as is $Q_n(2)$ for any $n$, but $Q_n(\ell)$ is generally not. 
However, the fact that $Q_n(\ell)$ dominates $Q_n(k)$ for all $1 \leq k \leq \ell$ may be seen as a near miss. Conjecture~\ref{conj:levels} is then just asking to what extent this phenomenon generalises. Note that the fact that $H$ is vertex transitive allows for easy relocations, but it also seems to be a necessary condition, as such relocations are not always possible in non-vertex-transitive reflection graphs such as the $1$-subdivision of an octahedron.

\appendix


\begin{thebibliography}{}

\bibitem{ACetal16}
O. Antol\'in Camarena, E. Cs\'oka, T. Hubai, G. Lippner and L. Lov\'asz, 
Positive graphs,
{\it European J. Combin.} {\bf 52} (2016), 290--301.

\bibitem{CFS10}
D. Conlon, J. Fox and B. Sudakov, An approximate version of Sidorenko's conjecture, {\it Geom. Funct. Anal.} {\bf 20} (2010), 1354--1366.

\bibitem{CKLL}
D. Conlon, J. H. Kim, C. Lee and J. Lee, Some advances on Sidorenko's conjecture, {\it J. Lond. Math. Soc.} {\bf 98} (2018), 593--608.

\bibitem{CL17}
{D. Conlon and J. Lee,} {Finite reflection groups and graph norms,} {\it Adv. Math.} {\bf 315} (2017), 130--165.    

\bibitem{CL19}
{D. Conlon and J. Lee,} {Sidorenko's conjecture for blow-ups,} {\it Discrete Anal.} 2021, Paper No. 2, 13 pp.

\bibitem{C22}
L. N. Coregliano, Left-cut-percolation and induced-Sidorenko bigraphs, preprint available at arXiv:2205.14703 [math.CO].

\bibitem{CR21}
L. N. Coregliano and A. A. Razborov, Biregularity in Sidorenko's conjecture, preprint available at arXiv:2108.06599 [math.CO].

\bibitem{ES82}
{P. Erd\H{o}s and M. Simonovits,} {Compactness results in extremal graph theory,} {\it Combinatorica} {\bf 2} (1982), 275--288.

\bibitem{GHL22}
F. Garbe, J. Hladk\'y and J. Lee, Two remarks on graph norms, {\it Discrete Comput. Geom.} {\bf 67} (2022), 919--929.

\bibitem{H10}
{H. Hatami,} {Graph norms and Sidorenko's conjecture,} {\it Israel J. Math.} {\bf 175} (2010), 125--150.

\bibitem{KLL16}
J. H. Kim, C. Lee and J. Lee, Two approaches to Sidorenko's conjecture, {\it Trans. Amer. Math. Soc.} {\bf 368} (2016), 5057--5074.

\bibitem{L19}
{J. Lee,} {On some graph densities in locally dense graphs,} {\it Random Structures Algorithms} {\bf 58} (2021), 322--344.

\bibitem{LS22}
J. Lee and A. Sidorenko, On graph norms for complex-valued functions, {\it J. Lond. Math. Soc.} {\bf 106} (2022),  1501--1538.

\bibitem{LiSz}
J. X. Li and B. Szegedy, On the logarithmic calculus and Sidorenko's conjecture, preprint available at arXiv:1107.1153 [math.CO].

\bibitem{L08}
{L. Lov\'asz,} {Graph homomorphisms: Open problems}, {unpublished manuscript available at http://www.cs.elte.hu/{\raise.17ex\hbox{$\scriptstyle\mathtt{\sim}$}}lovasz/problems.pdf}, 2008.

\bibitem{L12}
L. Lov\'asz, {\bf Large networks and graph limits}, Amer.~Math.~Soc.~Colloq.~Publ., 60, American Mathematical Society, Providence, RI, 2012.

\bibitem{S93}
{A. Sidorenko,} {A correlation inequality for bipartite graphs,} {\it Graphs Combin.} {\bf 9} (1993), 201--204.

\bibitem{S932}
{A. Sidorenko,} {Inequalities for functionals generated by bipartite graphs,} {\it Discrete Math. Appl.} {\bf 2} (1993), 489--504.

\bibitem{S20}
A. Sidorenko, Weakly norming graphs are edge-transitive, {\it Combinatorica} {\bf 40} (2020), 601--604.

\bibitem{Sunpub}
A. Sidorenko, personal communication.

\bibitem{S84}
M. Simonovits, 
Extremal graph problems, degenerate extremal problems, and supersaturated graphs, in Progress in graph theory (Waterloo, Ont., 1982), 419--437, Academic Press, Toronto, ON, 1984.

\bibitem{Sz14}
B. Szegedy, An information theoretic approach to Sidorenko's conjecture, preprint available at arXiv:1406.6738 [math.CO].

\end{thebibliography}
\end{document}